\documentclass[12pt]{amsart}

\usepackage{amssymb}
\usepackage{amsfonts}
\usepackage{amsthm}
\usepackage{graphicx}
\usepackage{color}
\usepackage{url}
\usepackage{microtype}
\usepackage{tikz}
\usepackage{enumitem}
\usepackage{hyperref}
\definecolor{darkblue}{RGB}{0,0,160}
\hypersetup{
colorlinks,%
citecolor=black,%
filecolor=black,%
linkcolor=darkblue,%
urlcolor=darkblue
}

\usepackage[utf8]{inputenc}
\usepackage[
text={445pt,573pt},
headheight=9pt,
centering
]{geometry}

\newcommand{\excise}[1]{}

\newtheorem{thm}{Theorem}[section]
\newtheorem{lemma}[thm]{Lemma}

\newtheorem{prop}[thm]{Proposition}

\newtheorem{alg}[thm]{Algorithm}
\newtheorem{rec}[thm]{Recipe}

\theoremstyle{definition}
\newtheorem{example}[thm]{Example}
\newtheorem{remark}[thm]{Remark}
\newtheorem{defn}[thm]{Definition}

\numberwithin{equation}{section}

\renewcommand\>{\rangle}
\newcommand\<{\langle}

\newcommand\Id{I}

\newcommand\kk{\Bbbk}

\DeclareMathOperator{\supp}{supp} 

\newcommand{\inD}[1][\relax]{\def\argone{#1}\def\temprelax{\relax}
  \ifx\argone\temprelax\right.\else\,\middle|#1\right.{}\fi}

\allowdisplaybreaks[1]

\begin{document}

\title{Detecting Binomiality}

\author{Carsten Conradi}
\address{Hochschule für Technik und Wirtschaft\\ Berlin, Germany}
\email{carsten.conradi@htw-berlin.de}

\author{Thomas Kahle}
\address{Otto-von-Guericke Universität\\ Magdeburg, Germany} 
\urladdr{\url{http://www.thomas-kahle.de}}

\thanks{The authors are supported by the research focus dynamical
systems of the state Saxony-Anhalt.}

\date{August 2015}

\makeatletter
  \@namedef{subjclassname@2010}{\textup{2010} Mathematics Subject Classification}
\makeatother



\begin{abstract}
Binomial ideals are special polynomial ideals with many
algorithmically and theoretically nice properties.  We discuss the
problem of deciding if a given polynomial ideal is binomial.  While
the methods are general, our main motivation and source of examples is
the simplification of steady state equations of chemical reaction
networks.  For homogeneous ideals we give an efficient, Gr\"obner-free
algorithm for binomiality detection, based on linear algebra only.  On
inhomogeneous input the algorithm can only give a sufficient condition
for binomiality.  As a remedy we construct a heuristic toolbox that
can lead to simplifications even if the given ideal is not binomial.
\end{abstract}

\maketitle
\setcounter{tocdepth}{1}
\tableofcontents

\section{Introduction}
\label{sec:intro}

Non-linear algebra is a mainstay in modern applied mathematics and
across the sciences.  Very often non-linearity comes in the form of
polynomial equations which are much more flexible than linear
equations in modeling complex phenomena.  The price to be paid is that
their mathematical theory---commutative algebra and algebraic
geometry---is much more involved than linear algebra.  Fortunately,
polynomial systems in applications often have special structures.  In
this paper we focus on \emph{sparsity}, that is, polynomials having
few terms.

The sparsest polynomials are monomials.  Systems of monomial equations
are a big topic in algebraic combinatorics, but in the view of
modeling they are not much help.  Their solution sets are unions of
coordinate hyperplanes.  The next and more interesting class are
\emph{binomial systems} in which each polynomial is allowed to have
two terms.  Binomials are flexible enough to model many interesting
phenomena, but sparse enough to allow a specialized
theory~\cite{eisenbud96:_binom_ideal}.  The strongest classical
results about binomial systems require one to seek solutions in
algebraically closed field such as the complex numbers.  However, for
the objects in applications (think of concentrations or probabilities)
this assumption is prohibitive.  One often works with non-negative
real numbers and this leads to the fields of real and semi-algebraic
geometry.  New theory in combinatorial commutative algebra shows that
for binomial equations field assumptions can be skirted and that the
dependence of binomial systems on their coefficients is quite
weak~\cite{kahle11mesoprimary}.  For binomial equations one can hope
for results that do not depend on the explicit values of the
parameters and are thus robust in the presence of uncertainty.

The main theme of this paper is how to detect binomiality, that is,
how to decide if a given polynomial system is equivalent to a binomial
system.  The common way to decide binomiality is to compute a
Gr\"obner basis since an ideal can be generated by binomials if and
only if any reduced Gr\"obner basis is
binomial~\cite[Corollary~1.2]{eisenbud96:_binom_ideal}.  For
polynomial systems arising in applications, however, computing a
reduced Gr\"obner basis is often too demanding: as parameter values
are unknown, computations have to be performed over the field of
rational functions in the parameters.  Even though this is
computationally feasible, it is time consuming and usually yields an
output that is hard to digest for humans.  This added complexity comes
from the fact that Gr\"obner bases contain a lot more information than
what may be needed for a specific task such as deciding binomiality of
a polynomial system.  Hence Gr\"obner-free methods are desirable.

\subsection*{Gr\"obner-free methods} 
Gr\"obner bases started as a generalization of Gauss elimination to
polynomials.  They have since come back to their roots in linear
algebra by the advent of F4 and F5 type algorithms which try to
arrange computations so that sparse linear algebra can
exploited~\cite{eder2014survey}.  Our method draws on linear algebra
in bases of monomials too, and is inspired by these developments in
computer algebra.

Deciding if a set of polynomials can be brought into binomial form
using linear algebra is the question whether the coefficient matrix
has a \emph{partitioning kernel basis} (Definition~\ref{d:pkb} and
Proposition~\ref{p:pkbvsBin}).  Deciding this property requires only
row reductions and hence is computationally cheap compared to
Gr\"obner bases.  It was shown in \cite{MillanToricSteady} that, if
the coefficient matrix of a suitably extended polynomial system admits
a partitioning kernel basis, then the polynomial system is generated
by binomials.  As a first insight we show that the converse of this
need not hold (Example~\ref{e:noConverse}).

In general computer algebra profits from homogeneity.  This is true
for Gr\"obner bases where, for example, Hilbert function driven
algorithms can be used to convert a basis for a term-order that is
cheap to compute into one for an expensive order, such as lex.  We
also observe this phenomenon in our Gr\"obner-free approach: a
satisfying answer to the binomial detection problem can be found if
the given system of polynomials is homogeneous.  In
Section~\ref{sec:homogenous-case} we discuss this case which
eventually leads to Algorithm~\ref{a:binomialDetect}.

In the inhomogeneous case things are more complicated.  Gr\"obner
basis computations can be reduced to the homogeneous case by an easy
trick.  Detection of binomiality can not
(Example~\ref{ex:homogGensNO}).  We address this problem by collecting
heuristic approaches that, in the best case, establish binomiality
without Gr\"obner bases (Recipe~\ref{recipe}).  Our approaches can
also be used if the system is not entirely binomial, but has some
binomials.  In Example~\ref{e:ERK} we demonstrate this on a polynomial
system from~\cite{fein-043}.

\subsection*{Binomial steady state ideals}

While binomiality detection can be applied to any polynomial system,
our motivation comes from chemical reaction network theory where
ordinary differential equations with polynomial right-hand sides are
used to model dynamic processes in systems biology.  The mathematics
of these systems is extremely challenging, in particular since
realistic models are huge and involve uncertain parameters.  As a
consequence of the latter, studying dynamical systems arising in
biological applications often amounts to studying parameterized
families of polynomial ODEs.  The first order of business (and concern
of a large part of the work in the area) is to determine steady states
which are thus the non-negative real zeros of families of
parameterized polynomial equations.  Moreover, the structure of the
polynomial ODEs entails the existence of affine linear subspaces that
are invariant for solutions. Hence questions concerning existence and
uniqueness of steady states or existence of multiple steady states are
equivalent to questions regarding the intersection of the zero set of
a parameterized family of polynomials with a family of affine linear
subspaces. 

If the polynomial equations describing steady states are equivalent to
binomial equations (that is, generate a binomial ideal), then their
mathematical analysis becomes much easier.  This is the main theme
of~\cite{MillanToricSteady}.  If a system is binomial, then, for
instance, one can decide efficiently if positive steady states exist.
If so, then a monomial parametrization can be found using only linear
algebra over the integers~\cite[Section~2]{eisenbud96:_binom_ideal},
and the steady states are \emph{toric}: they are the positive real
part of a toric variety.  A sufficient criterion for toric steady
states appears in \cite[Theorem~4.1]{johnston2014translated}.  Since
zero sets of general polynomial systems need not have parametrizations
at all, we view the task of detecting binomiality as an important step
in analyzing polynomials---in systems biology, or other areas like
algebraic statistics, control theory, economics, etc.

A frequent and challenging problem in the analysis of dynamical
systems in biology is to decide \emph{multistationarity}, that is, the
existence of parameter values leading to more than one isolated steady
state.  A variety of results for precluding multistationarity has
appeared in recent years.  See, for instance,
\cite{joshi2012simplifying,wiuf2013power,craciun2006multiple,banaji2010graph}
for methods employing the Jacobian.  Similarly, several sufficient
criteria for multistationarity have emerged (for
example~\cite{craciun2010multiple,fein-043}).  In general this problem
remains very hard.  However, in the case of binomial steady state
equations, the question of multistationarity can often be answered
effectively, for example positively
by~\cite[Theorem~5.5]{MillanToricSteady}, or negatively
by~\cite[Theorem~1.4]{muller2013sign}.  Both of these results require
only the study of systems of linear inequalities.

\subsection*{Notation}
In this paper we work with the polynomial ring $\kk[x_1,\dots,x_n]$ in
$n$ variables.  The coefficient field $\kk$ is usually $\mathbb{R}$,
or the field of real rational functions in a set of parameters.  Our
methods are agnostic towards the field.  A system of polynomial
equations $f_1 = f_2 = \dots = f_s = 0$ in the variables
$x_1,\dots,x_n$ is encoded in the \emph{ideal}
$\<f_1,\dots,f_s\> \subset \kk[x_1,\dots,x_n]$.  A polynomial is
\emph{homogeneous} if all its terms have the same total degree, and an
ideal is homogeneous if it can be generated by homogeneous
polynomials.  A \emph{binomial} is a polynomial with at most two
terms.  In particular, a monomial is a binomial.  It is important to
distinguish between binomial ideals and binomial systems.  A
\emph{binomial system} $f_1 = \dots = f_s = 0$ is a polynomial system
such that each $f_i$ is a binomial.  In contrast, an ideal
$\<f_1,\dots,f_s\>$ is a \emph{binomial ideal} if there exist
binomials that generate the same ideal.  Thus general non-binomials do
not form a binomial system, even if they generate a binomial ideal.
For the sake of brevity we will not give an introduction to
commutative algebra here, but refer to standard text books
like~\cite{cox96:_ideal_variet_algor,singularbook}.  The very modest
amount of matroid theory necessary in Section~\ref{sec:gb-free} can be
picked up from the first pages of~\cite{oxley2006matroid}.

\subsection*{Acknowledgment}
\label{sec:acknowledgement}

We thank David Cox for discussions on the role of homogeneity in
computer algebra.  We are grateful to Alicia Dickenstein for pointing
out a crucial error in an earlier version of this paper.  TK~is
supported by CDS the Center for Dynamical Systems at Otto-von-Guericke
University.

\section{Gröbner-free criteria for binomiality}
\label{sec:gb-free}

The most basic criterion to decide if an ideal is binomial is to
compute a Gröbner basis.  This works because the Buchberger algorithm
is binomial-friendly: an S-pair of binomials is a binomial.  Since the
reduced Gröbner basis is unique and must be computable from the
binomial generators, it consists of binomials if and only if the ideal
is binomial.  However, Gr\"obner bases can be very hard to compute, so
other criteria using only linear algebra are also desirable.  Linear
algebra enters, when we write a polynomial system as $A\Psi(x)$, the
product of a coefficient matrix $A$ with entries in $\kk$, and a
vector of monomials $\Psi(x)$.  Clearly, if we use row operations on
the matrix to bring it into a form where each row has at most two
non-zero entries, then the ideal is generated by binomials and
monomials.  This criterion is too naive to detect all binomial ideals
since it allows only $\kk$-linear combinations of the given
polynomials.  We show in Section~\ref{sec:homogenous-case} that, at
least for homogeneous ideals, it can be extended to a
characterization.  Before we embark into the details, we formalize the
condition on the matrix.

\begin{defn}\label{d:pkb}
A matrix $A$ has a \emph{partitioning kernel basis} if its kernel
admits a basis of vectors with disjoint supports, that is, if there
exists a basis $b^{(1)}, \dots, b^{(d)}$ of $\ker (A)$ such that
$\supp(b^{(i)}) \cap \supp (b^{(j)}) = \emptyset$ for all $i\neq j$.
\end{defn}

The following proposition allows one to check for a partitioning
kernel basis with linear algebra.  The underlying reason is the very
restricted structure of the kernel, expressed best in matroid
language.

\begin{prop}\label{p:pkb}
The following are equivalent for any matrix~$A$.
\begin{enumerate}[label=\arabic*),ref=\arabic*]
\item\label{it:pkb} $A$ has a partitioning kernel basis.
\item\label{it:directSum} The column matroid of $A$ is a direct sum of
uniform matroids $U_{r-1,r}$ of corank one, and possibly several
coloops~$U_{1,1}$.
\item\label{it:echelon} The reduced row echelon form of $A$ has at
most two non-zero entries in each row.
\end{enumerate}
\end{prop}
\begin{proof}
\ref{it:pkb} $\Rightarrow$ \ref{it:directSum}: Let $b_1,\dots,b_k$ be
the partitioning kernel basis.  The supports of this basis satisfy the
circuit axioms and are thus equal to the circuits of the column
matroid of~$A$.  Indeed, non-containment and circuit elimination are
satisfied trivially because there is no overlap between any two
circuits.  For any non-zero element $\tilde{b} \in \ker (A)$, we have
$\tilde{b} = \sum_i \lambda_i b_i$.  By the partitioning kernel basis
property
$\supp(\tilde{b}) = \bigcup_i \{\supp (b_i) : \lambda_i \neq 0\}$, so
either $\tilde {b}$ is proportional to one of the $b_i$, or its
support properly contains the support of a circuit, so it cannot be a
circuit.  The columns of $A$ which do not appear in any circuit are
coloops and the remaining columns form a direct sum of $k$ uniform
matroids of corank one.

\ref{it:directSum} $\Rightarrow$ \ref{it:echelon}: If the column
matroid of $A$ is a direct sum of matroids, then the unique reduced
row echelon form has block structure corresponding to the direct sum
decomposition.  Therefore it suffices to consider a single block which
has one-dimensional kernel of full support (the coloops are
$(1\times 1)$-identity blocks).  Ignoring zero rows, the reduced row
echelon form of such a matrix is $(\Id_{r-1} | c)$ where $r-1$ is the
rank and $c \in\kk_{\neq 0}^{r-1}$.

\ref{it:echelon} $\Rightarrow$ \ref{it:pkb}: Rows of the reduced row
echelon form with exactly one non-zero entry correspond to positions
where every element of the kernel has a zero.  Thus we can assume that
there are none and each row of $A$ has exactly two non-zero entries.
Let $c$ be a non-pivotal column with $r-1$ non-zero entries.  The
restriction of $A$ to $c$ and the corresponding pivotal columns yields
a block containing $(I_{r-1}|c)$ and some zero rows.  The unique
kernel vector corresponding to the dependencies in this block is
orthogonal to the kernel of the remaining columns.  This procedure
can be applied to any non-pivotal column.  The thus constructed basis
is a partitioning kernel basis.
\end{proof}

\begin{remark}\label{r:pkbComplexity}
Proposition~\ref{p:pkb} shows that the complexity of deciding if a
matrix has a partitioning kernel basis is essentially the same as that
of Gauss-Jordan elimination.  One needs $O(n^3)$ field operations
where $n$ is the larger of the dimensions of the matrix.
\end{remark}

\begin{remark}\label{r:partitionMatroid}
A direct sum of (arbitrary) uniform matroids is called a partition
matroid.
\end{remark}

We now translate Proposition~\ref{p:pkb} to polynomial systems.

\begin{prop}\label{p:pkbvsBin}
If $A$ has a partitioning kernel basis, and $\Psi(x)$ is a vector of
monomials of appropriate length, then the ideal
$\<A\Psi(x)\> \subset \kk[x_1,\dots,x_n]$ is binomial.  If $A\Psi(x)$
is any system that can be transformed into a binomial system using
only $\kk$-linear combinations, then $A$ has a partitioning kernel
basis.
\end{prop}
\begin{proof}
Up to coloops the first part is \cite[Theorem~3.3]{MillanToricSteady}
and the coloops only give monomials.  The second statement is clear
since $\kk$-linear combinations of polynomials are row operations on
the coefficient matrix and those do not change the kernel.
\end{proof}

Our general strategy is to suitably extend a given system $A\Psi(x)$
with redundant polynomials such that Proposition~\ref{p:pkbvsBin}
yields binomiality of an extended system $A'\Psi'(x)$.  This happens
in the following example.

\begin{example}\label{e:notNecessary}
Let $f_1 = x - y$, $f_2 = z - w$, and
$f_3 = x (f_1 + f_2) = x^2 - xy + xz - xw$.  Ordering the monomials
$\Psi^T = (x,y,z,w,x^2,xy,xz,xw)$, the system linearizes as
\[
f = (f_1,f_2,f_3) =
\begin{pmatrix}
1 & -1 & 0 & 0 & 0  & 0  &  0  &  0     \\
0 & 0 & 1 & -1 & 0  & 0  &  0  &  0     \\
0 & 0 & 0 & 0 &  1  & -1 &  1  &  -1 
\end{pmatrix}
\cdot \Psi.
\]
The coefficient matrix is in reduced row echelon form and does not
have a partitioning kernel basis by Proposition~\ref{p:pkb}.
Algorithm~\ref{a:binomialDetect} takes this problem into account,
working degree by degree.  The ideal is binomial since $f_3$ is a
binomial (in fact, zero) in the quotient ring
$\kk[x,y,z]/\<x-y,z-w\>$.
\end{example}

The first hint into how to extend $A$ and $\Psi(x)$ is the following
theorem due to P\'erez Mill\'an et al.

\begin{thm}[{\cite[Theorem~3.19]{MillanToricSteady}}]
\label{t:sufficient}
Let $f_1 = f_2 = \cdots = f_s = 0$ be a polynomial system.  If there
exist monomials $x^{\alpha_1}, \dots, x^{\alpha_m}$ such that, for
some $i_1,\dots,i_m \in [s]$, the system
\begin{equation} \label{eq:extramons}
f_1 = f_2 = \cdots = f_s = x^{\alpha_1}f_{i_1} = \cdots =
x^{\alpha_m}f_{i_{m}} = 0
\end{equation}
has a coefficient matrix with a partitioning kernel basis, then
$\<f_1,\dots,f_s\>$ is binomial.
\end{thm}
Theorem~\ref{t:sufficient} is true since the additional generators
in~\eqref{eq:extramons} do not change the ideal that the system
generates.  This together with the explicit description of the
binomial generators in the case of a partitioning kernel
basis~\cite[Theorem~3.3]{MillanToricSteady} yields the result.  If the
condition in Theorem~\ref{t:sufficient} were also necessary, then a
test for binomiality could be built on trying to systematically
identify the monomials~$x^\alpha$.  However, the converse of
Theorem~\ref{t:sufficient} is not true.
\begin{example}\label{e:noConverse}
Let $I = \<f_1,f_2\>$ be the homogeneous binomial ideal generated by
the non-binomials $f_1 = x-y+x^2+y^2+z^2$, $f_2 = x^2+y^2+z^2$.  For
no choice of monomials $m_{11}, \dots, m_{1r}$,
$m_{21}, \dots, m_{2s} \in \kk[x,y,z]$ does the
coefficient matrix of the system
\begin{equation}\label{eq:f1f2system}
f_1 = m_{11}f_{1} = \cdots = m_{1r} f_{1} =  
f_2 = m_{21}f_{2} = \cdots = m_{2s} f_{2} =  0
\end{equation}
have a partitioning kernel basis.
\end{example}
For the proof Example~\ref{e:noConverse} we first need the following
curious little fact.
\begin{lemma}\label{l:noBinomial}
The ideal $I = \<x^2+y^2+z^2\> \subset \kk[x,y,z]$ does not contain a
non-zero binomial.
\end{lemma}
\begin{proof}
We can assume that $\kk$ is algebraically closed, since if $I$
contains a non-zero binomial, then so does its extension to the
algebraic closure.  Assume that for some $f\in\kk[x,y,z]$, the product
$b = f(x^2+y^2+z^2)\in I$ is a binomial.  We can assume that $b$ is
not divisible by any variable.  Indeed, if a variable divides~$b$,
then it divides $f$ and we find a lower degree binomial in~$I$.  Since
$I$ is homogeneous, we can also assume that $b$ is homogeneous.
Potentially renaming the variables, we can assume
$b = x^d - \lambda y^sz^{d-s}$ for some $0\le s < d$ and
$\lambda\in\kk$.  Since $\kk$ is algebraically closed, there is a
solution $\xi$ to the equation $x^2 = -1$.  The generator
$x^2+y^2+z^2$ vanishes at $(\xi,1,0)$ but $b$ does not vanish there.
This contradiction shows that $I$ cannot contain a binomial.
\end{proof}
\begin{proof}[Proof of Example~\ref{e:noConverse}]
Let $d$ be the highest total degree among monomials in the
system~\eqref{eq:f1f2system} and consider the restriction of all
involved polynomials to degree~$d$.  Since the highest degree part of
both $f_1$ and $f_2$ equals $(x^2 + y^2 + z^2)$, only monomial
multiples of $(x^2+y^2+z^2)$ can contribute to the degree $d$ part of
the system~\eqref{eq:f1f2system}.  If the whole coefficient matrix
of~\eqref{eq:f1f2system} had a partitioning kernel basis, then also
the submatrix with only the columns for degree $d$ monomials had one
(Proposition~\ref{p:pkb}).  In this case row reductions on the
submatrix would yield a binomial in degree~$d$ in the ideal generated
by $(x^2 + y^2 + z^2)$. This is impossible by
Lemma~\ref{l:noBinomial}.
\end{proof}

Example~\ref{e:noConverse} may seem contrived, but this kind of
``trivial obfuscation'' of binomials does happen in applications.  Of
course, for humans it is obvious that one should first isolate the
linear binomial $x-y$ and then search for implied quadratic binomials
which reduce the trinomial.  Our next aim is
Algorithm~\ref{a:binomialDetect} which implements this idea, at least
in the homogeneous case.  The homogeneity assumption cannot be
skirted, unfortunately.  It is true that an ideal is binomial if and
only if its homogenization is
binomial~\cite[Corollary~1.4]{eisenbud96:_binom_ideal}, but the
homogenization is not accessible without a Gröbner basis.  It would be
superb for our purposes if homogenizing the generators of a binomial
ideal would always yield a binomial ideal.  Unfortunately this is not
the case as Example~\ref{ex:homogGensNO} shows.

\section{The homogeneous case}
\label{sec:homogenous-case}

If a given ideal $I$ is homogeneous, the graded vector space structure
of the quotient $\kk[x_1,\dots,x_n]/I$ allows one to check binomiality
degree by degree.  For this we need some basic facts about quotients
modulo binomials (see \cite[Section~1]{eisenbud96:_binom_ideal} for
details).  Any set of binomials $B$ in $\kk[x_1,\dots,x_n]$ induces an
equivalence relation on the set of monomials in $\kk[x_1,\dots,x_n]$
under which $m_1 \sim m_2$ if and only if
$m_1 - \lambda m_2 \in \<B\>$ for some non-zero $\lambda \in \kk$.  As
a $\kk$-vector space the quotient ring $\kk[x_1,\dots,x_n]/\<B\>$ is
spanned by the equivalence classes of monomials and those are all
linearly independent~\cite[Proposition~1.11]{eisenbud96:_binom_ideal}.
If the binomials are homogeneous, then the situation is particularly
nice.  For example, the equivalence classes are finite and elements of
the quotient have well-defined degrees.  The notions of monomial,
binomial, and polynomial are extended to the quotient ring.  For
example, a binomial is a polynomial that uses at most two equivalence
classes of monomials.  The unified mathematical framework to treat
quotients modulo binomials are monoid algebras, but we refrain from
introducing this notion here.

As a consequence of the discussion above, a polynomial system
$f_1 = \dots = f_s = 0$ can be considered modulo binomials, and the
coefficient matrix of the quotient system is well-defined.  It arises
from the coefficient matrix of the original system by summing columns
for monomials in the same equivalence class.

\begin{example}\label{e:modBinom}
In $\kk[x,y]$, let $b = x^2-y^2$.  Among monomials of total degree
three, $x^3$ and $xy^2$, as well as $x^2y$ and $y^3$ become equal in
$\kk[x,y]/\<b\>$.  Thus the degree three part in the quotient is
two-dimensional with one basis vector per equivalence class.
Consequently, the trinomial $f=x^3+xy^2+y^3$ maps to a binomial with
coefficient matrix $[2,1]$.  This matrix arises from the matrix
$[1,1,1,0]$ by summing the columns corresponding to $x^3$ and~$xy^2$,
as well as those for $x^2y$ and~$y^3$.
\end{example}

The reduction modulo lower degree binomials in
Example~\ref{e:modBinom} can be done in general.  

\begin{lemma}\label{l:onedegree}
Let $f_1,\dots,f_s \in \kk[x_1,\dots,x_n]$ be homogeneous polynomials
of degree~$d$, and $B \subset \kk[x_1,\dots,x_n]$ a set of homogeneous
binomials of degree at most~$d$.  Then in the quotient ring
$\kk[x_1,\dots,x_n]/\<B\>$ the ideal $\<f_1,\dots,f_s\>/\<B\>$ is
binomial if and only if the coefficient matrix of the images of
$f_1,\dots,f_s$ in $\kk[x_1,\dots,x_n]/\<B\>$ has a partitioning
kernel basis.
\end{lemma}
\begin{proof}
The graded version of Nakayama's lemma (see Corollary~4.8b together
with Exercise~4.6 in~\cite{eisenbud95:_commut_algeb}) implies that the
ideal $\<f_1,\dots,f_s\>/\<B\> \subset \kk[x_1\dots,x_n]/\<B\>$ has a
well-defined number of minimal generators in each degree.  Therefore
any minimal generating set consists only of degree $d$ polynomials and
Proposition~\ref{p:pkbvsBin} applied to the finite-dimensional vector
space of degree $d$ polynomials correctly decides binomiality.
\end{proof}

Lemma~\ref{l:onedegree} is the basis for the following binomial
detection algorithm.

\begin{alg}\label{a:binomialDetect} $ $\\
\textbf{Input:} Homogeneous polynomials $f_1,\dots, f_s \in \kk[x_1,\dots,x_n]$.\\
\textbf{Output:} \emph{Yes} and a binomial generating set of
$\<f_1,\dots,f_s\>$ if one exists, \emph{No} otherwise.

\begin{enumerate}[label=\arabic*),ref=\arabic*]
\item Let 
\begin{itemize}
\item $B := \emptyset$,
\item $R := \kk[x_1,\dots,x_n]$,
\item $F := \{f_1,\dots,f_s\}$.
\end{itemize}
\item \label{it:sublist} While $F$ is not empty,
\begin{enumerate}[label=\alph*),ref=\alph*]
\item Let $F_{\text{min}}$ be the set of elements of minimal degree
in~$F$.
\item Redefine $F := F \setminus F_{\text{min}}$.
\item \label{it:computeEchelon} Compute the reduced row echelon form
$A$ of the coefficient matrix of~$F_{\text{min}}$.
\item \label{it:NO} If $A$ has a row with three or more non-zero
entries, output \emph{No} and stop.
\item \label{it:findBinomials} Find a set $B'$ of binomials in
$\kk[x_1,\dots,x_n]$ whose images in $R$ generate $\<F_{\text{min}}\>$
and redefine $B := B \cup B'$.
\item \label{it:quotient} Redefine $R := \kk[x_1,\dots,x_n] / \<B\>$.
\item \label{it:redefine} Redefine $F$ as its image in~$R$.
\end{enumerate}
\item Output \emph{Yes} and~$B$.
\end{enumerate}
\end{alg}
\begin{proof}[Proof of correctness and termination]
Termination is obvious.  In fact, the maximum number of iterations in
the while loop equals the number of distinct total degrees among
$f_1,\dots,f_s$.  Step~\ref{it:sublist}.\ref{it:NO} relies on
Proposition~\ref{p:pkb}.  In
step~\ref{it:sublist}.\ref{it:findBinomials}, binomials that generate
$\<F_{\text{min}}\>$ in~$R$ can be read off the reduced row echelon
form via Proposition~\ref{p:pkb}.  Then any preimages in
$\kk[x_1,\dots,x_n]$ of those binomials suffice for~$B'$.
Lemma~\ref{l:onedegree} shows that the while loop either exhausts $F$
if $\<f_1,\dots,f_s\>$ is binomial, or stops when this is not the
case.
\end{proof}

\begin{remark}
In the homogeneous case there is a natural choice of
finite-dimensional vector spaces to work in: polynomials of a fixed
degree.  In each iteration of the while loop in
Algorithm~\ref{a:binomialDetect}, the rows of $A$ span the vector
space of polynomials of degree $d$ in the ideal (modulo the binomials
in~$\<B\>$).  In the general inhomogeneous situation extra work is
needed to construct a suitable finite-dimensional vector space.  In
particular, one needs to select from the infinite list of binomials in
the ideal not too many, but enough to reduce all given polynomials to
binomials whenever this is possible.  An interesting problem for the
future is to adapt one of the selection strategies from the F4
algorithm for Gr\"obner bases~\cite{faugere1999new} for this task.
\end{remark}

\begin{remark}
Coefficient matrices of polynomial systems are typically very sparse.
An efficient implementation of Algorithm~\ref{a:binomialDetect} has to
take this into account.
\end{remark}

\begin{remark}\label{r:multiply}
Algorithm~\ref{a:binomialDetect} could also be written completely in
the polynomial ring without any quotients.  Then in each new degree,
one would have to consider the coefficient matrix of $F_\text{min}$
together with all binomials of degree $d$ in the ideal~$\<B\>$.  This
list grows very quickly and so does the list of monomials appearing in
these binomials.  Thus it is not only more elegant to work with the
quotient, but also more efficient.
\end{remark}

To implement Algorithm~\ref{a:binomialDetect} completely without
Gröbner bases some refinements are necessary.  Simply using
$R = \kk[x_1,\dots,x_n] / B$ in \textsc{Macaulay2} will make it
compute a Gröbner basis of~$B$ to effectively work with the quotient.
For our purposes, however, this is not necessary.

\begin{prop}
Algorithm~\ref{a:binomialDetect} can be implemented without Gröbner
bases.
\end{prop}
\begin{proof}
The critical step is \ref{it:sublist}.\ref{it:redefine}, when the
algorithm reduces $F$ modulo the binomials already found.  For the
following step \ref{it:sublist}.\ref{it:computeEchelon} elements of
$F_\text{min}$ need to be written in terms of a basis of the
finite-dimensional vector space~$R_{\deg (F_\text{min})}$ of degree
$\deg (F_\text{min})$ monomials modulo the binomials in~$B$.  The
equivalence relation introduced in the beginning of this section can
also be thought of as a graph on monomials, and thus these reductions
can be carried out with graph enumeration algorithms like breadth
first search.  Restricting to monomials of degree
$\deg (F_\text{min})$, the connected components are a vector space
basis of $R_{\deg (F_\text{min})}$ and can thus be used to gather
coefficients in step~\ref{it:sublist}.\ref{it:redefine}.
\end{proof}

\begin{remark}
The feasibility of graph-theoretic computations in cases where Gröbner
bases cannot be computed has been demonstrated in
\cite{kahle12:positive-margins}.  Example~4.9 there contains a
binomial ideal whose Gröbner basis cannot be computed, but whose
non-radicality was proved using a graph-theoretic computation.  This
yielded a negative answer to the question of radicality of conditional
independence ideals in algebraic statistics.
\end{remark}

\begin{remark}
Using Gröbner bases one represents each connected component of the
graph defined by $\<B\>$ by its least monomial with respect to the
term order.  Our philosophy is that this is not necessary: one should
work with the connected components per se.  Why bother with picking
and finding a specific representative in each component if any
representative works?  In an implementation one could choose a data
structure that for each monomial stores an index of the connected
component it belongs to.
\end{remark}

\begin{remark}\label{r:many-examples}
It is trivial to generate classes of examples where Gr\"obner bases
methods fail, but Algorithm~\ref{a:binomialDetect} is quick.  For
example, take any set of binomials whose Gr\"obner basis cannot be
computed and add any polynomial in the ideal.
Algorithm~\ref{a:binomialDetect} immediately goes to work on reducing
the polynomial modulo the binomials, while any implementation of
Gr\"obner bases embarks into its hopeless task.
\end{remark}

\begin{remark}\label{r:no-gb-wanted}
Remark~\ref{r:many-examples} highlights the Gr\"obner-free spirit of
our method.  The Gr\"obner basis of an ideal contains much more
information than binomiality.  One should avoid expensive computation
to decide this simple question.
\end{remark}

\section{Heuristics for the inhomogeneous case}
\label{sec:inhomog}
The ideals one encounters in chemical reaction network theory are
often not homogeneous, so that the results from
Section~\ref{sec:homogenous-case} do not apply.  The first idea one
may have for the inhomogeneous case is to work with some (partial)
homogenization.  Gr\"obner bases are quite robust in relation to
homogenization.  For example, to compute a Gr\"obner basis of a
non-homogeneous ideal it suffices to homogenize the generators,
compute a Gr\"obner basis of this homogeneous ideal, and then
dehomogenize.  Although the intermediate homogeneous ideal is
generally not equal to the homogenization of the original ideal, the
dehomogenized Gr\"obner basis is a Gr\"obner basis of the
dehomogenized ideal \cite[Exercise~1.7.8]{singularbook}.

Unfortunately the notion of binomiality does not lend itself to that
kind of tricks.  Geometrically, homogenizing (all polynomials in) an
ideal yields the projective closure and dehomogenizing restricts to
one affine piece.  Homogenizing only the generators creates extra
components \emph{at infinity} and these components need not be
binomial.  Even if they are binomial, the intersection need not be
binomial (see also~\cite[Problem~17.1]{kahle11mesoprimary}).  This is
the case in the following example.
\begin{example}\label{ex:homogGensNO}
The ideal $\<ab-x, ab - y, x+y+1\> \subset \kk[a,b,x,y]$ is binomial
as it equals $\<2y+1, 2x+1, 2ab + 1\>$.  Homogenizing the generators,
however, yields the non-binomial ideal $\<ab-xz, ab-yz, x+y+z\>$.
\end{example}

We now present some alternatives that do not give complete answers but
are quick to check.  They can be applied before resorting to an
expensive Gr\"obner basis computation.

The quickest (but least likely to be successful) approach is to try
linear algebraic manipulations of the given polynomials.  Equivalently
one applies row operations to the coefficient matrix, for instance,
computing the reduced row echelon form.  If it has a partitioning
kernel basis, then the ideal is binomial and all non-binomial
generators are $\kk$-linear combinations of the binomials.  While it
may seem very much to ask for this, it does happen for the family of
networks in \cite[Section~4]{MillanToricSteady}.

If just linear algebra is not successful, one can homogenize the
generators and run Algorithm~\ref{a:binomialDetect}.  If the resulting
homogeneous ideal comes out binomial, then the original ideal was
binomial by the following simple fact, proven for instance
in~\cite[Corollary~A.4.16]{singularbook}.
\begin{prop}\label{p:homogguess}
Let $I \subset \kk[x_1,\dots,x_n]$ be an ideal and
$I' \subset \kk[x_0,x_1,\dots,x_n]$ the homogeneous ideal generated by
the homogenizations of the generators of~$I$ (using variable $x_0)$.
Then $I$ is generated by the dehomogenization of any generating set
of~$I'$.
\end{prop}

We now illustrate a phenomenon leading to failure of the above
heuristics.
\begin{example}\label{e:shinar}
Consider the network from \cite[Example~3.15]{MillanToricSteady}.  The
steady states are non-negative real zeros of the following
polynomials.
\begin{align*}
  f_1 &=-k_{12} x_{1}+k_{21} x_{2}-k_{1112} x_{1}
	x_{7}+(k_{1211}+k_{1213}) x_{9}, \\ 
  f_2 &= k_{12} x_{1}-k_{21} x_{2}-k_{23} x_{2}+k_{32} x_{3}+k_{67}
	x_{6}, \\ 
  f_3 &= k_{23} x_{2}-k_{32} x_{3}-k_{34} x_{3}-k_{89} x_{3}
	x_{7}+k_{910} x_{8}+k_{98} x_{8}, \\ 
  f_4 &= k_{34} x_{3}-k_{56} x_{4} x_{5}+k_{65} x_{6}, \\ 
  f_5 &= -k_{56} x_{4} x_{5}+k_{65} x_{6}+k_{910} x_{8}+k_{1213} x_{9},
  \\ 
  f_6 &= k_{56} x_{4} x_{5}-(k_{65}+k_{67}) x_{6}, \\ 
  f_7 &= k_{67} x_{6}-k_{1112} x_{1} x_{7}-k_{89} x_{3} x_{7}+k_{98}
	x_{8}+k_{1211} x_{9}, \\ 
  f_8 &= k_{89} x_{3} x_{7}-(k_{910}+k_{98}) x_{8}, \\ 
  f_9 &= k_{1112} x_{1} x_{7}-(k_{1211}+k_{1213}) x_{9}.
\end{align*}
The binomials $f_6$, $f_8$, and $f_9$ can be used to eliminate one of
every pair $(x_6, x_4x_5)$, $(x_8, x_3x_7)$, and $(x_9,x_1x_7)$.  We
eliminate $x_4x_5, x_8$, and~$x_9$.  It can be checked that
eliminating $x_1x_7$ instead of $x_9$ does not lead to binomials
immediately (although it leads to linear trinomials).
\begin{align*}
  f_1' &=  -k_{12} x_{1}+k_{21} x_{2}, \\
  f_2' &= k_{12} x_{1}- (k_{21} + k_{23}) x_{2}+k_{32}
	 x_{3}+ k_{67}x_6, \\
  f_3' &= k_{23} x_{2}-(k_{32}+k_{34}) x_{3}, \\
  f_4' &=  k_{34} x_{3}- k_{67}x_6, \\
  f_5' &= - k_{67}x_6
	 +\frac{k_{1112} k_{1213} x_{1}
	 x_{7}}{k_{1211}+k_{1213}}+\frac{k_{89} k_{910} x_{3}
	 x_{7}}{k_{910}+k_{98}}, \\ 
  f_7' &= k_{67}x_6 
	 -\frac{(k_{1112} k_{1213} (k_{910}+k_{98})
	 x_{1}+(k_{1211}+k_{1213}) k_{89} k_{910} x_{3})
	 x_{7}}{(k_{1211}+k_{1213}) (k_{910}+k_{98})}.
\end{align*}
Using the linear relations $f_1'$ and $f_3'$ the remaining system is
recognized to consist of only two independent binomials:
\begin{align*}
  f_2'' &= - k_{34}x_3 +
	  k_{67}x_6 \\
  f_4'' &= k_{34}x_3 -
	  k_{67}x_6, \\
  f_5'' &= - k_{67}x_6
	  +\left(\frac{k_{1112} k_{1213} k_{21}}{k_{12}
	  (k_{1211}+k_{1213})}+\frac{k_{23} k_{89}
	  k_{910}}{(k_{32}+k_{34}) (k_{910}+k_{98})}\right) x_{2} x_{7}, \\
  f_7'' & = k_{67}x_6 
	  +\left(-\frac{k_{1112} k_{1213}
	  k_{21}}{k_{12} (k_{1211}+k_{1213})}-\frac{k_{23} k_{89}
	  k_{910}}{(k_{32}+k_{34}) (k_{910}+k_{98})}\right) x_{2} x_{7}.
\end{align*}
This analysis shows that the steady state ideal under consideration
equals the binomial ideal
$\<f_1', f_2'', f_3', f_5'', f_6, f_8, f_9\>$.  The Gr\"obner basis
computation in \cite[Example~3.15]{MillanToricSteady} also yields the
result, but it is arguably less instructive.  Note also that naive
homogenization does not yield binomiality.  The element $f_2$ is
linear.  After homogenization, Algorithm~\ref{a:binomialDetect} would
pick only this element as $F_\text{min}$ and stop since it is not a
binomial.
\end{example}

The effect in Example~\ref{e:shinar} motivates our final method: term
replacements using known binomials.  We expect this to be very useful
in applications from system biology for the following reasons.
\setlist{nolistsep}
\begin{itemize}
\item It often happens that non-binomial generators are linear
combinations of binomials as in Example~\ref{e:shinar} where
$f_1 = f_1' + f_9$.
\item Steady state ideals of networks with enzyme-substrate complexes
always have some binomial generators.  These complexes are produced by
only one reaction and thus their rate of change is binomial.
\item In MAPK networks, which describe certain types of cellular
signaling, one often finds binomials of the form $k x_ax_b - k'x_c$.
\item Frequently binomials in steady state ideals are linear.
Equivalently some of the concentrations are equal up to a scaling
(which may depend on kinetic parameters).  This happens for all
examples in~\cite{millan2014mapk}.
\end{itemize}

We now illustrate term replacements in a larger example which comes
from the network for ERK activation embedded in two negative feedback
loops (see ~\cite[Section~5]{fein-043} for pointers to the relevant
biology).
\begin{example}\label{e:ERK}
Consider the following steady state ideal generated by 29 polynomials.
\footnotesize
\begin{align*}
  f_{1} &= -k_{1} x_{1} x_{2}+k_{2} x_{3}+k_{6} x_{6}, \qquad 
	  f_{2} = -k_{1} x_{1} x_{2}+k_{2} x_{3}+k_{3} x_{3}, \qquad 
	  f_{3} = k_{1} x_{1} x_{2}-k_{2} x_{3}-k_{3} x_{3}, \\ 
  f_{4} &= k_{11} x_{10}+k_{12} x_{10}+k_{38} x_{25}+k_{42} x_{27}+k_{3} x_{3}-k_{37} x_{18} x_{4}, \\
	&\qquad -k_{4}x_{4} x_{5} +k_{5} x_{6}-k_{7} x_{4} x_{7}+k_{8} x_{8}+k_{9} x_{8}-k_{10} x_{4} x_{9}\\
  f_{5} &= k_{14} x_{12}+k_{15} x_{12}+k_{17} x_{13}+k_{18} x_{13}+k_{35} x_{24}+k_{36} x_{24}+k_{41} 
	  x_{27}+k_{42} x_{27}-k_{13} x_{11} x_{5}-k_{34} x_{16} x_{5} \\
	&\qquad -k_{40} x_{26} x_{5}-k_{4} x_{4}
	  x_{5}+k_{5} x_{6}+k_{6} x_{6}-k_{16} x_{5} x_{9}, \\
  f_{6} &= k_{4} x_{4} x_{5}-k_{5} x_{6}-k_{6} x_{6}, \qquad 
	  f_{7} = k_{18} x_{13}-k_{7} x_{4} x_{7}+k_{8} x_{8}, \qquad
	  f_{8} = k_{7} x_{4} x_{7}-k_{8} x_{8}-k_{9} x_{8}, \\
  f_{9} &= k_{11} x_{10}+k_{15} x_{12}+k_{17} x_{13}+k_{9} x_{8}-k_{10} x_{4} x_{9}-k_{16} x_{5}
	  x_{9}, \qquad 
	  f_{10} = -k_{11} x_{10}-k_{12} x_{10}+k_{10} x_{4} x_{9}, \\
  f_{11} &= k_{12} x_{10}+k_{14} x_{12}-k_{19} x_{11} x_{14}+k_{20} x_{15}+k_{21} x_{15}-k_{22}
	   x_{11} x_{16}+k_{23} x_{17}+k_{24} x_{17}-k_{13} x_{11} x_{5}, \\
  f_{12} &= -k_{14} x_{12}-k_{15} x_{12}+k_{13} x_{11} x_{5}, \qquad
	   f_{13} = -k_{17} x_{13}-k_{18} x_{13}+k_{16} x_{5} x_{9}, \\
  f_{14} &= -k_{19} x_{11} x_{14}+k_{20} x_{15}+k_{30} x_{21}+k_{36} x_{24}, \qquad
	   f_{15} = k_{19} x_{11} x_{14}-k_{20} x_{15}-k_{21} x_{15}, \\
  f_{16} &= k_{21} x_{15}-k_{22} x_{11} x_{16}+k_{23} x_{17}-k_{28} x_{16} x_{19}+k_{27} x_{20}+k_{29}
	   x_{21}+k_{33} x_{23}+k_{35} x_{24}-k_{34} x_{16} x_{5}, \\
  f_{17} &= k_{22} x_{11} x_{16}-k_{23} x_{17}-k_{24} x_{17}, \\
  f_{18} &= k_{24} x_{17}-k_{25} x_{18} x_{19}+k_{26} x_{20}-k_{31} x_{18} x_{22}+k_{32} x_{23}+k_{38}
	   x_{25}+k_{39} x_{25} \\
	&\qquad -k_{43} x_{18} x_{28}+k_{44} x_{29}+k_{45} x_{29}-k_{37} x_{18} x_{4}, \\
  f_{19} &= -k_{46} x_{19}-k_{28} x_{16} x_{19}-k_{25} x_{18} x_{19}+k_{26} x_{20}+k_{27} x_{20}+k_{29}
	   x_{21}+k_{30} x_{21}+k_{45} x_{29}, \\
  f_{20} &= k_{25} x_{18} x_{19}-k_{26} x_{20}-k_{27} x_{20}, \qquad
	   f_{21} = k_{28} x_{16} x_{19}-k_{29} x_{21}-k_{30} x_{21}, \\
  f_{22} &= -k_{31} x_{18} x_{22}+k_{32} x_{23}+k_{33} x_{23}, \qquad
	   f_{23} = k_{31} x_{18} x_{22}-k_{32} x_{23}-k_{33} x_{23}, \\
  f_{24} &= -k_{35} x_{24}-k_{36} x_{24}+k_{34} x_{16} x_{5}, \qquad
	   f_{25} = -k_{38} x_{25}-k_{39} x_{25}+k_{37} x_{18} x_{4}, \\
  f_{26} &= k_{39} x_{25}+k_{41} x_{27}-k_{40} x_{26} x_{5}, \qquad
	   f_{27} = -k_{41} x_{27}-k_{42} x_{27}+k_{40} x_{26} x_{5} \\
  f_{28} &= k_{46} x_{19}-k_{43} x_{18} x_{28}+k_{44} x_{29}, \qquad
	   f_{29} = k_{43} x_{18} x_{28}-k_{44} x_{29}-k_{45} x_{29}.
\end{align*}
\normalsize

After some obvious factorization, the following elements are
binomials: $f_{2}$, $f_{6}$, $f_{8}$, $f_{10}$, $f_{12}$, $f_{13}$,
$f_{15}$, $f_{17}$, $f_{20}$, $f_{21}$, $f_{22}$, $f_{23}$, $f_{24}$,
$f_{25}$, $f_{27}$, $f_{29}$.  The system has seven conservation
relations, which can be found by linear algebra.  According to our
strategy to use binomials to simplify the system, we eliminate, if
possible, non-binomials using the conservation relations.  This is not
always possible, as some of the conservation relations stem from
duplicate equations like $f_2 = -f_3$.  We eliminate
$f_3,f_4,f_5,f_8,f_9,f_{18}$, and~$f_{19}$.  The remaining
non-binomial part consists of $f_1$, $f_7$, $f_{11}$, $f_{14}$,
$f_{16}$, $f_{26}$, and~$f_{28}$.  Dividing by reaction constants,
each of the binomials is of the form $x_i = K x_jx_l$ for some
rational expression $K$ involving only reaction constants.  Using
these in the non-binomials yields 

\footnotesize
\begin{align*}
  f'_{1} &= -\frac{k_{1} k_{3} x_{1} x_{2}}{k_{2}+k_{3}}+\frac{k_{4}
	  k_{6} x_{4} x_{5}}{k_{5}+k_{6}}, \qquad 
  f'_{7} = -\frac{k_{7} k_{9} x_{4} x_{7}}{k_{8}+k_{9}}+\frac{k_{16} k_{18} x_{5} x_{9}}{k_{17}+k_{18}},
  \\
  f'_{11} & = 
	   - \frac{ k_{13} k_{15} x_{11} x_{5}}{k_{14}+k_{15}} + \frac{k_{10} k_{12} x_{4} x_{9}}{k_{11}+k_{12}}, \\
  f'_{14} &= -\frac{k_{19} k_{21} x_{11} x_{14}}{k_{20}+k_{21}}+\frac{k_{28} k_{30} x_{16} x_{19}}{k_{29}+k_{30}}+\frac{k_{34}
	   k_{36} x_{16} x_{5}}{k_{35}+k_{36}}, \\
  f'_{16} &= \frac{k_{19} k_{21} x_{11} x_{14}}{k_{20}+k_{21}}-\frac{k_{22} k_{24} x_{11} x_{16}}{k_{23}+k_{24}}-k_{28}
	   x_{16} x_{19}+\frac{k_{28} k_{29} x_{16}
	   x_{19}}{k_{29}+k_{30}}+\frac{k_{25} k_{27} x_{18} x_{19}}{k_{26}+k_{27}} \\ 
	& \qquad + \frac{k_{31} k_{33} x_{18}  x_{22}}{k_{32}+k_{33}} -k_{34} x_{16} x_{5}+\frac{k_{34} k_{35} x_{16} x_{5}}{k_{35}+k_{36}},
  \\
  f'_{26} &= \frac{k_{37} k_{39} x_{18} x_{4}}{k_{38}+k_{39}}-\frac{k_{40} k_{42} x_{26} x_{5}}{k_{41}+k_{42}},
	   \qquad
	   f'_{28} = k_{46} x_{19}-\frac{k_{43} k_{45} x_{18} x_{28}}{k_{44}+k_{45}}.
\end{align*}
\normalsize

In particular, we find five new binomials $f'_1$, $f'_7$, $f'_{11}$,
$f'_{26}$, and~$f'_{28}$.  Adding $f'_{14}$ to $f'_{16}$ yields the
trinomial
\begin{equation*}
f_{16}'' = -\frac{k_{22} k_{24} x_{11} x_{16}}{k_{23}+k_{24}} +
\frac{k_{25} k_{27} x_{18} x_{19}}{k_{26}+k_{27}} + 
\frac{k_{31} k_{33} x_{18}  x_{22}}{k_{32}+k_{33}}.
\end{equation*}
Consequently, the original system is equivalent to a system consisting
of 27 binomials and two trinomials of a relatively simple shape.  For
comparison we computed the Gr\"obner basis in \textsc{Macaulay2}
with rational functions in the reaction rates as coefficients.
Although the computation finished in just 18 minutes, the result is
practically unusable.  The Gr\"obner basis consists of 169 elements
each of it with huge rational functions as coefficients.  The
structure that we observed above is completely lost.
\end{example}

The lesson learned from Example~\ref{e:ERK} is that term replacements
using binomials are useful in solving a polynomial system, even if the
end result is not binomial.  Especially in the non-homogeneous case
where the notion of minimal generators is absent, computations with
the \textsc{Binomials} package~\cite{kahle11:binom-jsag} in
\textsc{Macaulay2}~\cite{M2} can probably only assist, but not
automatically do useful reductions.  For example, a natural general
choice would be to replace higher degree monomials by lower degree
ones, but this would not directly reveal binomiality in
Example~\ref{e:shinar}.

Finally, we summarize a possible strategy to deal with inhomogeneous
ideals.  Example~\ref{ex:homogGensNO}, for instance, is solved already
by item~\ref{it:try1}, but also by item~\ref{it:try3}.
\begin{rec}\label{recipe}$ $\\[-2ex]
\begin{enumerate}[label=\arabic*),ref=\arabic*]
\item\label{it:try1} Try linear algebra and
Proposition~\ref{p:pkbvsBin}.
\item Homogenize the given ideal and run
Algorithm~\ref{a:binomialDetect}.  If the algorithm returns binomials,
then by Proposition~\ref{p:homogguess} the original (dehomogenized)
ideal is binomial.  The homogenization should be carried out after
linear algebra reductions to possibly detect homogeneity already at an
earlier stage (compare Example~\ref{e:noConverse}).
\item\label{it:try3} If Algorithm~\ref{a:binomialDetect} returns a
negative answer, dehomogenize and use known binomials for term
replacements (as in Example~\ref{e:ERK}).  Potentially homogenize
again with an enlarged generating set.
\item Compute a reduced Gr\"obner basis.
\end{enumerate}
\end{rec}


\begin{thebibliography}{21}
\providecommand{\enquote}[1]{``#1''}
\providecommand{\url}[1]{{\tt #1}}
\providecommand{\href}[2]{#2}

\bibitem[1]{banaji2010graph}
Banaji M., Craciun G., Graph-theoretic criteria for injectivity and unique equilibria in general chemical reaction systems, Adv. in Appl. Math. 44 (2010) 168 -- 184.
\bibitem[2]{singularbook}
Bachmann O., Greuel G.-M., Lossen C., Pfister G., Sch{\"o}nemann H.,
A~{S}ingular introduction to commutative algebra, Springer, Berlin, 2007.
\bibitem[3]{craciun2006multiple}
Craciun G., Feinberg M., Multiple equilibria in complex chemical reaction networks: {II}. {T}he species-reaction graph, SIAM J. Appl. Math. 66 (2006) 1321 -- 1338.
\bibitem[4]{craciun2010multiple}
Craciun G., Feinberg M., Multiple equilibria in complex chemical reaction networks: {S}emiopen mass action systems, SIAM J. Appl. Math. 70 (2010) 1859 -- 1877.
\bibitem[5]{fein-043}
Conradi C., Flockerzi D., {Multistationarity in mass action networks with applications to ERK activation}, J. Math. Biol. 65 (2012) 107 -- 156.
\bibitem[6]{cox96:_ideal_variet_algor}
Cox D. A., Little J. B., O'Shea D., Ideals, Varieties, and Algorithms, Springer, New York, 1996.
\bibitem[7]{eder2014survey}
Eder C., Faug{\`e}re J.-C., A survey on signature-based {G}r{\"o}bner
basis computations, preprint, arXiv:1404.1774.
\bibitem[8]{eisenbud96:_binom_ideal}
Eisenbud D., Sturmfels B., Binomial Ideals, Duke Math. J. 84 (1996) 1 -- 45.
\bibitem[9]{eisenbud95:_commut_algeb}
Eisenbud D., Commutative Algebra with a View Toward Algebraic Geometry, Springer, New York, 1995.
\bibitem[10]{faugere1999new}
Faug{\`e}re J.-C., A new efficient algorithm for computing {G}r{\"o}bner bases {(F4)}, J. Pure Appl. Algebra 139 (1999) 61 -- 88.
\bibitem[11]{M2}
Grayson D. R., Stillman M. E., {M}acaulay2, a software system for research in algebraic geometry, Available at \url{http://www.math.uiuc.edu/Macaulay2/}.
\bibitem[12]{joshi2012simplifying}
Joshi B., Shiu A., Simplifying the {J}acobian criterion for precluding multistationarity in chemical reaction networks, SIAM J. Appl. Math. 72 (2012) 857 -- 876.
\bibitem[13]{johnston2014translated}
Johnston M. D., Translated chemical reaction networks, Bull. Math. Biol. 76 (2014) 1081 -- 1116.
\bibitem[14]{kahle11mesoprimary}
Kahle T., Miller E., Decompositions of Commutative Monoid Congruences and Binomial Ideals, Algebra Number Theory 8 (2014) 1297 -- 1364.
\bibitem[15]{kahle12:positive-margins}
Kahle T., Rauh J., Sullivant S., Positive margins and primary decomposition, J. Commut. Algebra 6 (2014) 173 -- 208.
\bibitem[16]{kahle11:binom-jsag}
Kahle T., Decompositions of Binomial Ideals, J. Softw. Algebra Geom. 4
(2012) 1 -- 5.
\bibitem[17]{muller2013sign}
M{\"u}ller S., Feliu E., Regensburger G., Conradi C., Shiu A.,
Dickenstein A., Sign conditions for injectivity of generalized
polynomial maps with applications to chemical reaction networks and
real algebraic geometry, Found. Comput. Math., published onlineu under doi:10.1007/s10208-014-9239-3 (2015).
\bibitem[18]{MillanToricSteady}
P{\'e}rez Mill{\'a}n M., Dickenstein A., Shiu A., Conradi C., Chemical Reaction Systems with Toric Steady States, Bull. Math. Biol. 74 (2012) 1027 -- 1065.
\bibitem[19]{millan2014mapk}
P{\'e}rez Mill{\'a}n M., Turjanski A. G., {MAPK}'s networks and their capacity for multistationarity due to toric steady states, Math. Biosci. 262 (2015) 125 -- 137.
\bibitem[20]{oxley2006matroid}
Oxley J. G., Matroid Theory, Oxford University Press, Oxford, 2006.
\bibitem[21]{wiuf2013power}
Wiuf C., Feliu E., Power-law kinetics and determinant criteria for the preclusion of multistationarity in networks of interacting species, SIAM J. Appl. Dyn. Syst. 12 (2013) 1685 -- 1721.

\end{thebibliography}

\end{document}